\documentclass[11pt]{amsart}

\usepackage{etex}
\reserveinserts{28}
\usepackage[all]{xy}

\usepackage{amsfonts}
\usepackage{amsthm}
\usepackage{enumerate}
\usepackage{graphicx}
\usepackage{mathrsfs}
\usepackage{bm}
\usepackage{cite}
\usepackage{amssymb,amsmath} % font used for R in Real numbers
\usepackage{makecell}
\usepackage{tikz}
\usepackage{pgf}
\usepackage{tikz}
\usetikzlibrary{patterns}
\usepackage{pgffor}
\usepackage{pgfcalendar}
\usepackage{pgfpages}
\usepackage{shuffle,yfonts}
\usepackage{mathtools}
\DeclareFontFamily{U}{shuffle}{}
\DeclareFontShape{U}{shuffle}{m}{n}{ <-8>shuffle7 <8->shuffle10}{}

\newcommand{\bfj}{{\bf{j}}}
\newcommand{\bfk}{{\bf{k}}}
\newcommand{\bfl}{{\bf{l}}}

\catcode`!=11
\let\!int\int \def\int{\displaystyle\!int}
\let\!lim\lim \def\lim{\displaystyle\!lim}
\let\!sum\sum \def\sum{\displaystyle\!sum}
\let\!sup\sup \def\sup{\displaystyle\!sup}
\let\!inf\inf \def\inf{\displaystyle\!inf}
\let\!cap\cap \def\cap{\displaystyle\!cap}
\let\!max\max \def\max{\displaystyle\!max}
\let\!min\min \def\min{\displaystyle\!min}
\let\!frac\frac \def\frac{\displaystyle\!frac}
\catcode`!=12

\let\oldsection\section
\renewcommand\section{\setcounter{equation}{0}\oldsection}

\allowdisplaybreaks

\DeclareMathOperator*{\dep}{dep}

\newcommand{\tB}{{\widetilde{B}}}

\def\R{\mathbb{R}}
\def\N{\mathbb{N}}\def\Z{\mathbb{Z}}
\def\Q{\mathbb{Q}}

\theoremstyle{plain}
\newtheorem{thm}{Theorem}[section]
\newtheorem{lem}[thm]{Lemma}
\newtheorem{cor}[thm]{Corollary}

\theoremstyle{definition}

\newtheorem{re}[thm]{Remark}
\newtheorem{exa}[thm]{Example}

 \allowdisplaybreaks
\usetikzlibrary{arrows,shapes,chains}
 \allowdisplaybreaks

\begin{document}
%%%%%%%%%%%%%%%%%%%% title %%%%%%%%%%%%%%%%%%%%%%%%%%%%%%%%%%%%%%%%%%%%%%%%
\title[Sums Involving Harmonic Numbers]{\bf Evaluation of Some Sums Involving Powers of Harmonic Numbers}

\thanks{
Email: cexu2020@ahnu.edu.cn (C. Xu), XixiZhang2019@163.com (X. Zhang), zhaoj@ihes.fr (J. Zhao: Corresponding author)}

\date{}
\maketitle

\begin{center}
 {\sc Ce Xu and Xixi Zhang}\\
 School of Mathematics and Statistics, Anhui Normal University, Wuhu 241002, P.R.C\\
 {\sc Jianqiang Zhao}\\
 Department of Mathematics, The Bishop's School, La Jolla,\small  CA 92037, USA
\end{center}

\bigskip
\noindent{\bf Abstract.} In this note, we extend the definition of multiple harmonic sums and apply their stuffle relations to obtain explicit evaluations of the sums $R_n(p,t)=\sum\nolimits_{m=0}^n m^p H_m^t$, where $H_m$ are harmonic numbers. When $t\le 4$ these sums were first studied by Spie\ss\ around 1990 and, more recently, by Jin and Sun. Our key step first is to find an explicit formula of a special type of the extended multiple harmonic sums. This also enables us to provide a general structural result of the sums $R_n(p,t)$ for all $t\ge 0$.



\medskip
\noindent{\bf Declarations.} 

\textbf{Funding:} Ce Xu is supported by the Scientific Research Foundation for Scholars of Anhui Normal University and the University Natural Science Research Project of Anhui Province (Grant No. KJ2020A0057)

\textbf{Conflicts of interest/Competing interests:} None.

\textbf{Availability of data and material:} N/A.

\textbf{Code availability:} N/A.

\medskip

\section{Introduction}
We begin with some basic notations. Let $\N$ (resp. $\Z$) be the set of positive integers (resp. integers) and $\N_0:=\N\cup \{0\}$.
A finite sequence $\bfk:=(k_1,\ldots, k_r)\in\N^r$ is called a \emph{composition}. We define its \emph{weight} and the \emph{depth} of $\bfk$, respectively, by
\begin{equation*}
 |\bfk|:=k_1+\cdots+k_r \quad \text{and}\quad \dep(\bfk):=r.
\end{equation*}

For any $n\in\N$ and $(k_1,\ldots,k_r)\in\Z^r$, we define the \emph{extended multiple harmonic sums} by
\begin{align}
H_n(k_1,\ldots,k_r):=\sum\limits_{n\geq n_1>\cdots>n_r>0 } \frac{1}{n_1^{k_1}\cdots n_r^{k_r}}\label{EMHSs}.
\end{align}
We set $H_n(\emptyset ):=1$ and $H_n(k_1,\ldots,k_r):=0$ if $n<r$. When $r=1$ and $k>0$, $H_n(k)=\sum\nolimits_{j=1}^n {1}/{j^k}$ is the $n$-th generalized harmonic number of order $k$, and furthermore, if $k=1$ then $H_n:=H_n(1)$ is the classical $n$-th harmonic number. Identities involving harmonic numbers and multiple harmonic sums (i.e., when all $k_j>0$) have been extensively studied in the literature (see, e.g., \cite{Zhao2016} and the references therein).

Let $(k_1,\ldots,k_r)\in\N^r$. In this note, we are particularly interested in the sums of the form
\begin{align}
H_n(-k_1,k_2,\ldots,k_r)=\sum\limits_{n\geq n_1>\cdots>n_r>0 } n_1^{k_1} \frac{1}{n_2^{k_2}\cdots n_r^{k_r}}=\sum_{m=1}^n m^{k_1}H_{m-1}(k_2,\cdots,k_r).\label{EMHSs-1}
\end{align}

Recall that there are two versions of Bernoulli numbers defined by the generating functions
\begin{equation}\label{equ:bern}
\frac{te^t}{e^t-1}=\sum_{n=0}^\infty B_n\frac{t^n}{n!} \qquad\text{and} \qquad
\frac{t}{e^t-1}=\sum_{n=0}^\infty \tB_n\frac{t^n}{n!}.
\end{equation}
It is well-known that $B_j=\tB_j$ if $j\ne1$, $B_1=1/2$ and $\tB_1=-1/2$.
By Faulhaber's formula,
\begin{align}\label{HN-N1}
H_n(-k)=\sum_{m=1}^n m^k=\frac{1}{k+1}\sum_{j=0}^k \binom{k+1}{j}B_jn^{k+1-j}
=n^k+\frac{1}{k+1}\sum_{j=0}^k \binom{k+1}{j}\tB_jn^{k+1-j}.
\end{align}

In \cite{S1990}, Spie\ss \ studied the summation
\begin{equation}\label{equ:Rnt}
R_n(d,t):=\sum_{m=0}^n m^d H_m^t\quad(n\in\N,d,t\in\N_0)
\end{equation}
and proved the following structure theorem (see \cite[Thm. 30]{S1990}).
\begin{thm}\label{thm1}
Let $p(m)$ be a polynomial in $m$ of degree $d$. Then for $t=1,2$ or $3$, there exist polynomials $q_0(n),\ldots,q_t(n)$ and $C(n)$ of degree at most $d+1$ such that
\begin{align}\label{OP1}
\sum_{m=0}^n p(m)H_m^t=\sum_{i=0}^t q_i(n)H_n^i+C(n)H_n(2),
\end{align}
for all nonnegative integers $n$. Moreover, $C(n)=0$ when $t=1,2$.
\end{thm}

Spie\ss\ also conjectured that Theorem \ref{thm1} holds for any positive integer $t\geq 4$. However, Jin and Sun \cite{JS2013} showed that the sum $R_n(0,4)=\sum\nolimits_{m=0}^n H_m^4$ cannot be represented by the form as conjectured by Spie\ss. Moreover, Jin and Sun \cite[Thms. 1.2 and 1.3]{JS2013} proved that
\begin{align}\label{OP2}
R_n(d,3)=\sum_{m=1}^n m^d H_m^3=H_n(-d)H_n^3+q_1(n)H_n^2+q_2(n)H_n+q_3(n)+\frac{\tB_d}{2}H_n(2),
\end{align}
where $q_1(n),q_2(n)$ and $q_3(n)$ are polynomial in $n$ of degree at most $d+1$. They also provided similar formulas for the sums
\begin{align}\label{OP3}
\sum_{m=0}^n p(m)H_m^4  \qquad\text{and} \qquad \sum_{m=0}^n m^dH_mH_m(2).
\end{align}
Nevertheless, the explicit formula of the polynomials $q_1(n),q_2(n),q_3(n)$ were not found.

The primary goal of this note is to establish explicit formulas of \eqref{OP2} and \eqref{OP3}. Our main idea is to express these sums using \eqref{EMHSs-1} for which
we are able to derive an explicit formula in general (see Theorem~\ref{thm-generalFormula}). At the end of the note, we present a result generalizing Theorem \ref{thm1} to all $t\ge 0$ (see Theorem \ref{thm:SpiessGeneral}).

\section{Explicit Formulas of Spie\ss's Results}
In this section we will study the sums of the form
\begin{equation}\label{equ:SpieSum}
\sum_{m=1}^n m^p H_{m-1}^t
\end{equation}
for $p,t\ge 0$. We remark that this is slightly different from $R_n(p,t)$ (see \eqref{equ:Rnt})
considered by Spie\ss \cite{S1990} but
they are closely related. See Remark \ref{rem:equivalent}.
The key idea to study \eqref{equ:SpieSum} is to express  $H_n^t$ by using general multiple harmonic
sums of the form \eqref{EMHSs-1} by applying the stuffle relations, also called quasi-shuffle relations (see
\cite{Hoffman2000}). In fact, more generally, we know
that for any composition $\bfk=(k_1,\ldots,k_r)$, the product $H_n(k_1)\cdots H_n(k_r)$ can be expressed in terms of a linear combination of multiple harmonic sums (for the explicit formula, see \cite[Eq. (2.4)]{WX2020}), for example
\begin{equation}\label{equ:H1H(2)}
H_nH_n(2)=H_n(1)H_n(2)=H_n(1,2)+H_n(2,1)+H_n(3).
\end{equation}

\subsection{Some Explicit Formulas of General MHSs}
To derive the general formula for $H_n(-p,\bfk)$ where $\bfk=(k_1,\dots,k_r)\in\N^r$,
we set $\bfk_i=(k_1,\dots,k_i)$ for all $1\le i\le r$ and $|\bfk_i|=k_0+\dotsm+k_i$ for all $0\le i\le r$.
\begin{thm}\label{thm-generalFormula}
Let $p\in\N_0$, $r\in\N$, and $k_1,\dots,k_r\in\N$. Put $k_0=j_0=0$, $k'_r=k_{r+1}=1$ and $k'_l=k_l$ for all $l<r$. Then
\begin{align}\label{equ:generalFormula}
&H_n(-p,k_1,\ldots,k_r)  =  \nonumber\\
-&\sum_{l=1}^r (-1)^l \left(\sum\limits_{0\le j_i\le p+i-|\bfk'_i|-|\bfj_{i-1}|\ \forall 1\le i\le l}
\prod_{h=0}^{l-1}\frac{\binom{p+h+1-|\bfk_h|-|\bfj_h|}{j_{h+1}}B_{j_{h+1}}}{p+h+1-|\bfk_h|-|\bfj_h|}\right)
n^{p+l-|\bfk_{l-1}|-|\bfj_l|} \cdot H_n(k_l,\dotsc,k_r) \nonumber\\
+&\sum_{l=1}^r (-1)^l \left(\sum_{\substack{p+l+1-|\bfk_l|\le |\bfj_l|\\ \le p+l-1-|\bfk_{l-1}|}}
\prod_{h=0}^{l-1}\frac{\binom{p+h+1-|\bfk_h|-|\bfj_h|}{j_{h+1}}B_{j_{h+1}}}{p+h+1-|\bfk_h|-|\bfj_h|}\right)
H_n\Big(|\bfk_l|+|\bfj_l|-l-p,k_{l+1},\dotsc,k_r\Big)  \nonumber\\
+&(-1)^r  \left(\sum\limits_{0\le j_l\le p+l-|\bfk_l|-|\bfj_{l-1}|\ \forall 1\le l\le r+1}
\prod_{h=0}^{r}\frac{\binom{p+h+1-|\bfk_h|-|\bfj_h|}{j_{h+1}}B_{j_{h+1}}}{p+h+1-|\bfk_h|-|\bfj_h|}\right)
 n^{p+r+1-|\bfk_r|-|\bfj_{r+1}|},
\end{align}
where we set the binomial coefficients $\binom{a}{b}=0$ if $b<0$.
\end{thm}
\begin{proof}
According to definition \eqref{EMHSs-1}, we have the following recurrence relation
\begin{align}\label{GMHS-EQR1}
&H_n(-p,k_1,\ldots,k_r)\nonumber\\
&=\sum_{m=1}^n m^p H_{m-1}(k_1,\ldots,k_r)=\sum_{m=1}^n m^p\sum\limits_{m>n_1>\cdots>n_r\geq 1}\frac{1}{n_1^{k_1}\cdots n_r^{k_r}}\nonumber\\
&=\sum_{m=1}^n m^p\left(H_{m}(k_1,\ldots,k_r)-\frac{H_{m-1}(k_2,\ldots,k_r)}{m^{k_{1}}}\right)\nonumber\\
&=\sum_{m=1}^n m^pH_{m}(k_1,\ldots,k_r)-\sum_{m=1}^n m^{p-k_{1}}H_{m-1}(k_2,\ldots,k_r)\nonumber\\
&=\sum_{m=1}^n m^p\sum\limits_{m\geq n_1>\cdots>n_r\geq 1} \frac{1}{n_1^{k_1}\cdots n_r^{k_r}}-\sum_{m=1}^n m^{p-k_{1}}H_{m-1}(k_2,\ldots,k_r)\nonumber\\
&=\sum_{n_{1}=1}^n\frac{H_{n_{1}-1}(k_2,\ldots,k_r)}{n_{1}^{k_{1}}}\sum_{m=n_{1}}^n m^p-\sum_{m=1}^n m^{p-k_{1}}H_{m-1}(k_2,\ldots,k_r)\nonumber\\
&=\sum_{n_{1}=1}^n\frac{H_{n_{1}-1}(k_2,\ldots,k_r)}{n_{1}^{k_{1}}}\left(\sum_{m=1}^n m^p-\sum_{m=1}^{n_{1}-1} m^p-n_{1}^{p}\right)\nonumber\\
&=H_n(-p)H_n(k_1,\ldots,k_r)-\sum_{n_{1}=1}^n\frac{H_{n_{1}-1}(k_2,\ldots,k_r)}{n_{1}^{k_{1}}}\sum_{m=1}^{n_{1}} m^p\nonumber\\
&=H_n(-p)H_n(k_1,\ldots,k_r)
-\sum_{j=0}^p\binom{p+1}{j}\frac{B_{j}}{p+1}\sum_{n_{1}=1}^n\frac{H_{n_{1}-1}(k_2,\ldots,k_r)}{n_{1}^{k_{1}}} n_{1}^{p+1-j}\nonumber\\
&=H_n(-p)H_n(k_1,\ldots,k_r)
- \sum_{j=0}^p\binom{p+1}{j}\frac{B_{j}}{p+1}H_n(k_1+j-p-1,k_{2},\ldots,k_r).
\end{align}
If $k_1>p+1$ in \eqref{GMHS-EQR1} then $k_1+j-p-1\in \N$, so $H_n(k_1+j-p-1,k_{2},\ldots,k_r)$ is the classical multiple harmonic sum. If $1\leq k_1\leq p+1$ then \eqref{GMHS-EQR1} can be rewritten in the following form
\begin{align}\label{GMHS-EQR3}
H_n(-p,\bfk)=H_n(-p)H_n(\bfk)
&-\sum_{j=p+2-k_{1}}^p\binom{p+1}{j}\frac{B_{j}}{p+1}H_n(k_1+j-p-1,k_{2},\ldots,k_r)\nonumber\\
&-\sum_{j=0}^{p+1-k_{1}}\binom{p+1}{j}\frac{B_{j}}{p+1}H_n(k_1+j-p-1,k_{2},\ldots,k_r).
\end{align}
When $r=1$ this implies the theorem by replacing $k$ by $p$ in \eqref{HN-N1}:
\begin{align}\label{GMHS-EQR2}
&H_n(-p,k_1)=H_n(-p)H_n(k_1)
-\sum_{j_1=p+2-k_{1}}^p\binom{p+1}{j_1}\frac{B_{j_1}}{p+1}H_n(k_1+j_1-p-1)\nonumber\\
&-\sum_{j_1=0}^{p+1-k_{1}}\sum_{j_2=0}^{p+1-k_1-j_1}
\binom{p+1}{j_1}\binom{p+2-k_1-j_1}{j_2}\frac{B_{j_1}B_{j_2}}{(p+1)(p+2-k_1-j_1)}n^{p+2-k_1-j_1-j_2}.
\end{align}
By applying the above recurrence relation, one can prove the theorem by induction on $r$ using \eqref{GMHS-EQR3}.
We leave the detail to the interested reader.
\end{proof}

To state our results more concisely, we set $j_0=a_1=0$ and for all $a_2,\dots,a_r\in\N_0$ we define the polynomials
\begin{equation*}
C^{(p)}_{a_2,\dots,a_r}(x):=
\sum\limits_{j_1+\dots+j_r\leq p-a_r \atop j_1,\dotsc,j_r\geq0}
\left(\prod_{i=1}^r \frac{\binom{p+1-a_i-j_1-\dotsm-j_{i-1}}{j_i}B_{j_i}}{p+1-a_i-j_1-\dotsm-j_{i-1}} \right) x^{p+1-a_r-j_1-\dotsm-j_r}.
\end{equation*}
In particular, we see that the subscript on the left-hand side is vacuous if $r=1$ and in this case $C^{(p)}(n)=H_n(-p)$.

\begin{cor}\label{cor-HN1}
For a composition $\bfk=(k_1,\ldots,k_r)$ and $p\in \N_0$, the sum $H_n(-p,\bfk)$ can be expressed in terms of a combination of products of polynomial in $n$ of degree $\leq p+1$ and multiple harmonic sums with depth $\leq r$. In particular, if $\bfk=(\{1\}_r)$ then we have
\begin{align}\label{GMHS-EQ1}
H_n(-p,\{1\}_{r})=H_n(-p)H_n(\{1\}_{r})+ \sum_{i=1}^r(-1)^{i}C^{(p)}_{\underbrace{\scriptstyle 0,\dotsc,0}_i}(n)H_n(\{1\}_{r-i}),
\end{align}
where $\{1\}_{r}$ means the string obtained by repeating $1$ exactly $r$ times.
\end{cor}
\begin{proof}
Observe that the $l$-th term in the second sum of \eqref{equ:generalFormula} is vacuous if $k_l=1$. The rest of the proof
is straight-forward.
\end{proof}

Setting $r=2$ in \eqref{GMHS-EQ1} and noting the fact that
\begin{equation}\label{equ:stufflezeta11}
H_n(1,1)=\frac{H_n^{2}-H_n(2)}{2},
\end{equation}
we get the following corollary.
\begin{cor} For $p\in \N_0$,
\begin{align}\label{GMHS-EQ1T}
H_n(-p,1,1)&=H_n(-p)H_n(\{1\}_2)-C^{(p)}_{0}(n)H_n+C^{(p)}_{0,0}(n) \notag\\
&=\frac12 H_n(-p)(H_n^{2}-H_n(2))-C^{(p)}_{0}(n)H_n+C^{(p)}_{0,0}(n).
\end{align}
\end{cor}

Taking $r=2$ in \eqref{equ:generalFormula} we can get the following corollary.
\begin{cor} For positive integers $k_1,k_2$ and nonnegative $p$, we have
\begin{align}\label{GMHS-k1k2}
&H_n(-p,k_1,k_2)=H_n(-p)H_n(k_1,k_2)-\frac{1}{p+1}\sum_{j=p+2-k_{1}}^p\binom{p+1}{j}B_{j}H_n(k_1+j-p-1,k_{2})\nonumber\\
&\quad+\frac{1}{p+1}\sum_{j_{1}=0}^{p+1-k_{1}}\sum_{j_{2}=p+3-k_{1}-k_{2}-j_{1}}^{p+1-k_{1}-j_{1}}
\frac{\binom{p+1}{j_{1}}\binom{p+2-k_{1}-j_{1}}{j_{2}}}{p+2-k_{1}-j_{1}}B_{j_{1}}B_{j_{2}}H_n(k_1+k_2+j_{1}+j_{2}-p-2)\nonumber\\
&\quad+\frac{1}{p+1}\sum_{j_{1}=0}^{p+1-k_{1}}\sum_{j_{2}=0}^{p+2-k_{1}-k_{2}-j_{1}}\sum_{j_{3}=0}^{p+2-k_{1}-k_{2}-j_{1}-j_{2}}
\frac{\binom{p+1}{j_{1}}\binom{p+2-k_{1}-j_{1}}{j_{2}}\binom{p+3-k_{1}-k_{2}-j_{1}-j_{2}}{j_{3}}}{(p+2-k_{1}-j_{1})(p+3-k_{1}-k_{2}-j_{1}-j_{2})}\nonumber\\
&\quad\quad\quad\quad\quad\quad\quad\quad\quad\quad\quad\quad\quad\quad\quad\quad\quad\quad\quad\times B_{j_{1}}B_{j_{2}}B_{j_{3}}n^{p+3-k_{1}-k_{2}-j_{1}-j_{2}-j_{3}}\nonumber\\
&\quad-\frac{1}{p+1}\sum_{j=0}^{p+1-k_{1}}\sum_{l=0}^{p+1-k_{1}-j}\frac{\binom{p+1}{j}\binom{p+2-k_{1}-j}{l}}{p+2-k_{1}-j}
n^{p+2-k_{1}-j-l}B_{j}B_{l}H_n(k_2).
\end{align}
\end{cor}

For $p\in \N_0$, put $D^{(p)}_{a_2,\dots,a_r}(B):=C^{(p)}_{a_2,\dots,a_r}(B)/B$ with $B^i$ understood as $B_i$
which is the convention in umbral calculus.
Setting $(k_1,k_2)=(1,2)$ and $(2,1)$ in \eqref{GMHS-k1k2}, respectively, we get the following two formulas:
\begin{align}\label{GMHS-p12}
H_n(-p,1,2)
&=H_n(-p)H_n(1,2)+D^{(p)}(B) H_n-C^{(p)}_{0}(n)H_n(2)+C^{(p)}_{0,1}(n),\\
H_n(-p,2,1)
&=H_n(-p)H_n(2,1)-B_{p}H_n(1,1)-C^{(p)}_{1}(n)H_n+C^{(p)}_{1,1}(n).\label{GMHS-p21}
\end{align}

For $p=0,1$ in \eqref{GMHS-p12} and \eqref{GMHS-p21}, we get
\begin{align*}
H_n(0,1,2)&=nH_n(1,2)-nH_n(2)+H_n,\\
H_n(0,2,1)&=nH_n(2,1)-\frac{H_n^{2}-H_n(2)}{2},\\
H_n(-1,1,2)&=\frac{n(n+1)}{2}H_n(1,2)-\frac{n(n+3)}{4}H_n^{2}+\frac{3}{4}H_n+\frac{1}{4}n,\\
H_n(-1,2,1)&=\frac{n(n+1)}{2}H_n(2,1)-\frac{1}{4}(H_n^{2}-H_n(2))-\frac{1}{2}nH_n+\frac{1}{2}n.
\end{align*}

\subsection{Explicit Formulas of Spie\ss's results}
We now present some explicit evaluations of Spie\ss's results on harmonic numbers (Theorem \ref{thm1}).

\begin{thm}\label{thm:Hn2} For $p\in\N_0$, we have
\begin{align}\label{GMHS-EQ5}
\sum_{m=1}^{n}m^{p}H_{m-1}^{2}=C^{(p)}(n)H_n^2-(2C_0^{(p)}(n)+B_p)H_n+2C_{0,0}^{(p)}(n)-C_1^{(p)}(n).
\end{align}
\end{thm}
\begin{proof} By the stuffle relation \eqref{equ:stufflezeta11} we have
\begin{equation*}
\sum_{m=1}^{n}m^{p}H_{m-1}^{2}=2H_n(-p,1,1)+H_n(-p,2).
\end{equation*}
Taking $k=2$ in \eqref{GMHS-EQR2} we see that
\begin{align*}
H_n(-p,2)=H_n(-p)H_n(2)- B_pH_n-C_1^{(p)}(n).
\end{align*}
Combining the above with \eqref{GMHS-EQ1T}, we arrive at \eqref{GMHS-EQ5} immediately.
\end{proof}

\begin{exa}
Letting $p=0,1$ in \eqref{GMHS-EQ5} we get
\begin{align*}
\sum_{m=1}^{n}H_{m-1}^{2}&=nH_n^{2}-(2n+1)H_n+2n,\\
\sum_{m=1}^{n}mH_{m-1}^{2}
&=\frac{n(n+1)}{2}H_n^{2}-\frac{n^{2}+3n+1}{2}H_n+\frac{n(n+5)}{4}.
\end{align*}
\end{exa}

\begin{thm} For $p\in\N_0$, we have
\begin{align}\label{GMHS-EQa}
\sum_{m=1}^{n}m^{p}H_{m-1}H_{m-1}(2)&=H_n(-p)H_nH_n(2)-\frac{B_p}{2}H_n^2+\bigg(D^{(p)}(B)-C_1^{(p)}(n)-\frac{p}{2}B_{p-1}\bigg)H_n\nonumber\\
&\quad-\bigg(C_0^{(p)}(n)+\frac{B_p}{2}\bigg)H_n(2)+C_{0,1}^{(p)}(n)+C_{1,1}^{(p)}(n)-C_2^{(p)}(n).
\end{align}
\end{thm}
\begin{proof} Applying identities \eqref{equ:H1H(2)}, \eqref{GMHS-EQR2}, \eqref{GMHS-p12} and \eqref{GMHS-p21}, we obtain the explicit evaluation of sum $\sum_{m=1}^{n}m^{p}H_{m-1}H_{m-1}(2)$ by a direct calculation. \end{proof}

\begin{exa}
Letting $p=0,1$ in \eqref{GMHS-EQa}, we get
\begin{align*}
\sum_{m=1}^{n}H_{m-1}H_{m-1}(2)
&=nH_nH_n(2)-\frac{1}{2}H_n^{2}-\frac{2n+1}{2}H_n(2)+H_n,\\
\sum_{m=1}^{n}mH_{m-1}H_{m-1}(2)
&=\frac{n(n+1)}{2}H_nH_n(2)-\frac{1}{4}H_n^{2}-\frac{n^{2}+3n+1}{4}H_n(2)+\frac{1-2n}{4}H_n+\frac{3}{4}n.
\end{align*}
\end{exa}

\begin{thm}\label{thm:Hn3} For $p\in\N_0$, we have
\begin{align*}
\sum_{m=1}^{n}m^{p}H_{m-1}^{3}=
&H_n(-p) H_n^3-3\bigg(C^{(p)}_{0}(n)+\frac{B_{p}}{2}\bigg)H_n^2+{\frac{B_{p}}{2}}H_n(2)\\
&+\bigg(6C^{(p)}_{0,0}(n)+3D^{(p)}(B)-3C^{(p)}_{1}(n)- \frac{p}{2} B_{p-1} \bigg)H_n \\
&-6C^{(p)}_{0,0,0}(n)+3C^{(p)}_{0,1}(n)+3C^{(p)}_{1,1}(n)-{C^{(p)}_2(n)}.
\end{align*}
\end{thm}
\begin{proof} By the stuffle relation we have
\begin{equation*}
\sum_{m=1}^{n}m^{p}H_{m-1}^{3}=6H_n(-p,\{1\}_3)+3H_n(-p,1,2)+3H_n(-p,2,1)+H_n(-p,3).
\end{equation*}
Taking $k=3$ in \eqref{GMHS-EQR2} we see that
\begin{align*}
H_n(-p,3)&=H_n(-p)H_n(3) - B_{p} H_n(2)- \frac{p}{2} B_{p-1}H_n
-\sum_{\substack{ j_1+j_2\le p-2\\ j_1,j_2\ge 0} }
{\frac{\binom{p+1}{j_1}\binom{p-1-j_1}{j_2}B_{j_1}B_{j_2}}{(p+1)(p-1-j_1)}n^{p-1-j_1-j_2}}\\
&=H_n(-p)H_n(3)- B_{p} H_n(2)- \frac{p}{2} B_{p-1}H_n
-{C^{(p)}_2(n)}.
\end{align*}
Combining the above with \eqref{GMHS-EQ1} (taking $r=3$),
\eqref{GMHS-p12} and \eqref{GMHS-p21}, we can complete the proof of the theorem
by a straight-forward simplification.
\end{proof}

\begin{exa} Taking $p=0,1,2$ in Theorem \ref{thm:Hn3} and using the relation
$$\sum_{m=0}^{n}m^{p}H_m^{3}= n^p H_m^{3}+\sum_{m=1}^{n} (m-1)^{p}H_{m-1}^{3}$$
we can confirm the first three examples in the last section of \cite{JS2013}.
\end{exa}

\begin{thm}\label{thm:Hn4}
Let $d\in\N_0$. Then for any polynomial $F(x)=\sum_{p=0}^d a_p x^p\in\Q[x]$ we have
\begin{equation*}
\sum_{m=1}^n F(m)  H_{m-1}^4 =H_n^4 \sum_{m=0}^n F(m)+\sum_{j=0}^3 Q_j(n)H_n^j
+P(n)H_n(2)+2F(B)H_n(2,1) + F(B)H_n(3)
\end{equation*}
where we should replace $B^i$ by $B_i$ for all $i$ in $F(B)$, and
$P(x)$ and $Q_j(x)$, $0\le j\le 3$, are all polynomials of $n$ of degree at most $d+1$
defined by
\begin{align*}
P(n)&=\sum_{p=0}^d a_p  \Big( -6D^{(p)}(B) +4D^{(p)}(n)+ \frac{p}{2}B_{p-1} \Big),\\
Q_0(n)&=\sum_{p=0}^d a_p  \Big(24C^{(p)}_{0,0,0,0}(n)-12C^{(p)}_{0,0,1}(n)-12C^{(p)}_{0,1,1}(n)-12C^{(p)}_{1,1,1}(n)\\
&\ \hskip2cm+4C^{(p)}_{0,2}(n)+6C^{(p)}_{1,2}(n)+4C^{(p)}_{2,2}(n)-C^{(p)}_{3}(n) \Big) ,\\
Q_1(n)&=\sum_{p=0}^d a_p   \Big(-24C^{(p)}_{0,0,0}(n)-12D^{(p)}_{0}(B)+12C^{(p)}_{0,1}(n)+12C^{(p)}_{1,1}(n) \\
&\ \hskip2cm +2\frac{d}{dx}D^{(p)}(x)\bigg|_{x=B}+ \frac{6(D^{(p)}(x)-B_p)}{x}\bigg|_{x=B}-4C^{(p)}_{2}(n)-\frac{p(p-1)}{6} B_{p-2}\Big) ,\\
Q_2(n)&=\sum_{p=0}^d a_p   \Big(12C^{(p)}_{0,0}(n)+6D^{(p)}(B)-6 C^{(p)}_{1}(n)- p B_{p-1} \Big) ,\\
Q_3(n)&=\sum_{p=0}^d a_p   \Big(-4 C^{(p)}_0(n)- 2B_p \Big).
\end{align*}
\end{thm}

\begin{proof}
We first consider $\sum_{m=1}^n m^p  H_{m-1}^4$ for any $p\in\N_0$.
Then by Theorem~\ref{thm-generalFormula}
\begin{align*}
H_n(-p,\{1\}_4)&=H_n(-p)H_n(\{1\}_4)-C^{(p)}_{0}(n) H_n(\{1\}_3)+C^{(p)}_{0,0}(n)H_n(\{1\}_2)-C^{(p)}_{0,0,0}(n)H_n+C^{(p)}_{0,0,0,0}(n),\\
H_n(-p,1,1,2)&=H_n(-p)H_n(1,1,2)-C^{(p)}_{0}(n)H_n(1,2)+C^{(p)}_{0,0}(n)H_n(2)-D^{(p)}_{0}(B)H_n-C^{(p)}_{0,0,1}(n),\\
H_n(-p,1,2,1)&=H_n(-p)H_n(1,2,1)-C^{(p)}_{0}(n)H_n(2,1)+ D^{(p)}(B)H_n(\{1\}_2)+C^{(p)}_{0,1}(n)H_n-C^{(p)}_{0,1,1}(n),\\
H_n(-p,2,1,1)&=H_n(-p)H_n(2,1,1)-B_pH_n(\{1\}_3)- C^{(p)}_{1}(n)H_n(\{1\}_2)+C^{(p)}_{1,1}(n) H_n-C^{(p)}_{1,1,1}(n),\\
H_n(-p,1,3)&=H_n(-p)H_n(1,3)- C^{(p)}_{0}(n)H_n(3)+D^{(p)}(n) H_n(2)+\frac12 \frac{d}{dx}D^{(p)}(x)\bigg|_{x=B} H_n + C^{(p)}_{0,2}(n),\\
H_n(-p,2,2)&=H_n(-p)H_n(2,2)-B_{p}H_n(1,2)-C^{(p)}_{1}(n) H_n(2)+\frac{D^{(p)}(x)-B_p}{x}\bigg|_{x=B} H_n +C^{(p)}_{1,2}(n),\\
H_n(-p,3,1)&=H_n(-p)H_n(3,1)- B_{p}H_n(2,1)- \frac{p}2 B_{p-1}H_n(\{1\}_2)-C^{(p)}_{2}(n)H_n+C^{(p)}_{2,2}(n),\\
H_n(-p,4)&=H_n(-p)H_n(4)-B_{p}H_n(3)-\frac{p}{2}B_{p-1}H_n(2)-\frac{p(p-1)}{6} B_{p-2}H_n-C^{(p)}_{3}(n).
\end{align*}
The theorem now follows directly from the three stuffle relations: \eqref{equ:stufflezeta11} and
\begin{align*}
H_n^3&=6H_n(\{1\}_3)+3H_n(1,2)+3H_n(2,1)+H_n(3),\\
H_n^4&=24H_n(\{1\}_4)+12\Big(H_n(1,1,2)+H_n(1,2,1)+H_n(2,1,1)\Big) \\
 &\qquad \qquad +6H_n(2,2)+4H_n(1,3)+4H_n(3,1)+H_n(4).
\end{align*}
This concludes the proof of the theorem.
\end{proof}

In order to see the consistence of our result with that of Jin and Sun in \cite{JS2013},
we may use the convention of umbral calculus (see, e.g., \cite{Roman}) to rewrite the
two versions of Bernoulli numbers defined by \eqref{equ:bern} as
\begin{equation*}
\frac{te^t}{e^t-1}=e^{B t},\qquad
\frac{t}{e^t-1}=e^{\tB t}
\end{equation*}
where in the expansion of the right-hand side we understand $B^i$ (resp. $\tB^i$) as $B_i$ (resp. $\tB_i$) for all $i\in\N_0$.
This notation scheme helps us to prove the next lemma in a straight-forward manner.

\begin{lem}\label{lem:twoBs}
Let $F(x)\in\R[x]$. Then $F(\tB)=0$ if and only if $F(B-1)=0$.
\end{lem}
\begin{proof} With umbral calculus notation we have
\begin{equation*}
\sum_{i=0}^\infty (B-1)^i \frac{t^i}{i!}=e^{(B-1)t}=e^{Bt}e^{-t}=\frac{te^t}{e^t-1}\cdot e^{-t}=\frac{t}{e^t-1}=\sum_{i=0}^\infty \tB^i \frac{t^i}{i!}.
\end{equation*}
Thus $(B-1)^i=\tB^i$ for all $i\in\N_0$. The lemma follows immediately.
\end{proof}

\begin{re}\label{rem:equivalent}
If $F(B)=0$ holds in Theorem~\ref{thm:Hn4} then we can recover an essentially equivalent form
of \cite[Theorem 1.3]{JS2013} by Lemma \ref{lem:twoBs} since
\begin{equation*}
\sum_{m=0}^n F(m)  H_m^4 =F(n)H_n^4+\sum_{m=1}^n F(m-1) H_{m-1}^4.
\end{equation*}
We point out that the Bernoulli numbers in \cite{JS2013} are denoted by $\tB_i$ in the current note.
\end{re}

\begin{exa}
By the examples at the end of \cite{JS2013}
\begin{equation*}
\text{ neither } \quad \sum_{k=0}^n (2k+1)H_k^4  \quad
\text{ nor } \quad
\sum_{k=0}^n (3k^2+k)H_k^4
\end{equation*}
involves $H_n(2,1)$ or $H_n(3)$. This can also follow easily from $2\tB_1+\tB_0=0$
and $3\tB_2+\tB_1=0$ by \cite[Theorem 1.3]{JS2013}. This is consistent with our result since
\begin{align}
\sum_{k=0}^n (2k+1)H_k^4=&(2n+1)H_n^4+ \sum_{m=1}^{n} (2m-1)H_{m-1}^4,  \label{equ:1Hn4}\\
\sum_{k=0}^n (3k^2+k)H_k^4=&(3n^2+n)H_n^4+ \sum_{m=1}^{n} (3m^2-5m+2)H_{m-1}^4,\label{equ:2Hn4}
\end{align}
and
$$2B_1-B_0=0, \qquad 3B_2-5B_1+2B_0=0.$$
Thus by Theorem \ref{thm:Hn4} we see that neither \eqref{equ:1Hn4} nor \eqref{equ:2Hn4}
involves $H_n(2,1)$ or $H_n(3)$.
\end{exa}

By looking back at Theorem~\ref{thm:Hn4} and Remark~\ref{rem:equivalent}, we find that not only Spie\ss's conjecture fails in general as shown by Jin and Sun
in \cite{JS2013} but also linear combinations of multiple harmonic sums with coefficients given by rational polynomials of $n$
should suffice to express $R_n(d,t)$ for all $d,t\in \N$. We now conclude our note by the following general theorem.
\begin{thm}\label{thm:SpiessGeneral}
Suppose $n\in\N$, $t\in\N_0$, and the polynomial $F(x)\in\Q[x]$ has degree $d$. Put $S_n(F)=\sum\nolimits_{m=1}^n F(m)$ and
denote by $V_n(d,t)$ the $\Q$-vector space generated by $P(n)H_n(\bfl)$ for all $P(x)\in\Q[x]$ of degree less than $d+2$ and
compositions $\bfl$ of depth less than $t$. Then we have
\begin{equation}\label{equ:SpiessGeneral}
\sum_{m=1}^n F(m) H_{m-1}^t, \sum_{m=0}^n F(m) H_{m}^t  \in S_n(F)H_n^t+V_n(d,t).
\end{equation}
\end{thm}

\begin{proof} Observe that
\begin{equation*}
\sum_{m=0}^n F(m)  H_m^t =F(n)H_n^t+\sum_{m=1}^n F(m-1) H_{m-1}^t.
\end{equation*}
It suffices to prove the theorem for $\sum_{m=1}^n F(m) H_{m-1}^t$, which can be further reduced to the case when $F(m)=m^p$
for some $p\le d$. By the stuffle relation have
\begin{equation*}
\sum_{m=1}^n m^p H_{m-1}^t -t!H_n(-p,\{1\}_t)=\sum_{\bfl: \dep(\bfl)<t, |\bfl|=t} c_\bfl H_n(-p,\bfl)
\end{equation*}
for some suitable integer coefficients $c_\bfl$. We see by Theorem~\ref{thm-generalFormula} that each term on the right-hand side of
the above lies inside $V_n(d,t)$. This concludes the proof of the theorem.
\end{proof}


\begin{thebibliography}{99}

\bibitem{Hoffman2000}
M. E. Hoffman, Quasi-shuffle products, J. Algebraic Combin. 11(2000), 49--68.

\bibitem{Roman}
S.\ Roman, The Umbral Calculus, Dover Publications, Reprint edition (April 17, 2019).

\bibitem{S1990}
J. Spie\ss, Some identities involving harmonic numbers, Math. Comp. 192(1990), 839--863.

\bibitem{JS2013}
H. Jin, L.H. Sun, On Spie\ss's conjecture on harmonic numbers, Discrete Appl. Math. 161(13-14)(2013), 2038--2041.

\bibitem{WX2020}
C. Xu, W. Wang, Explicit formulas of Euler sums via multiple zeta values, J. Symb. Comput. 101(2020), 109--127.

\bibitem{Zhao2016}
J. Zhao, Multiple Zeta Functions, Multiple Polylogarithms and Their Special Values, Series on Number Theory and Its Applications: Volume 12, World Scientific Publishing, 2016.

\end{thebibliography}
\end{document}